\documentclass[11pt]{article}
\usepackage[utf8]{inputenc}
\def\final{1}
\usepackage{color}
\usepackage{subfigure}
\usepackage{url}
\usepackage{graphicx}
\usepackage{amsmath}
\usepackage{amsthm}
\usepackage{bm}
\usepackage[plainpages=false]{hyperref}

\usepackage{titlesec}
\titleformat{\section}
{\normalfont\large\bfseries\filcenter}{\thesection.}{1 ex}{}
\titleformat{\subsection}[runin]
{\normalfont\normalsize\bfseries\filcenter}{\thesubsection.}{1 ex}{}

\usepackage[charter]{mathdesign}
\usepackage[mathcal]{eucal}

\usepackage[margin=1.2 in]{geometry}

\usepackage[square,numbers,sort&compress]{natbib}

\ifnum\final=0
\newcommand{\mynote}[1]{\marginpar{\tiny\sf #1}}
\else
\newcommand{\mynote}[1]{}
\fi

\usepackage{theomac}
\newtheoremWithMacro{theorem}{Theorem}
\newtheoremWithMacro{lemma}{Lemma}[section]
\newtheorem{question}{Question}

\newtheoremWithMacro{corollary}{Corollary}
\newtheoremWithMacro{prop}[lemma]{Proposition}

\newcommand{\figref}[1]{Figure \ref{fig:#1}}

\newcommand{\lemref}[1]{Lemma \ref{lemma:#1}}

\newcommand{\propref}[1]{Proposition \ref{prop:#1}}
\newcommand{\theoref}[1]{Theorem \ref{theo:#1}}
\newcommand{\secref}[1]{Section \ref{sec:#1}}
\renewcommand{\eqref}[1]{(\ref{eq:#1})}

\newcommand{\lemlab}[1]{\label{lemma:#1}}
\newcommand{\proplab}[1]{\label{prop:#1}}
\newcommand{\theolab}[1]{\label{theo:#1}}
\newcommand{\seclab}[1]{\label{sec:#1}}
\newcommand{\eqlab}[1]{\label{eq:#1}}

\newcommand{\R}{\mathbb{R}}
\newcommand{\Z}{\mathbb{Z}}
\newcommand{\HH}{\mathrm{H}}
\newcommand{\eop}{\hfill$\qed$}
\newcommand{\bgamma}{\bm{\gamma}}
\begin{document}
\title{Henneberg constructions and covers of cone-Laman graphs}
\author{Louis Theran\thanks{Institut für Mathematik,
Diskrete Geometrie, Freie Universität Berlin, \url{theran@math.fu-berlin.de}}}
\date{}
\maketitle
\begin{abstract}
\begin{normalsize}
We give Henneberg-type constructions for three families of sparse colored graphs
arising in the rigidity theory of periodic and forced symmetric frameworks.  The proof
method, which works with Laman-sparse finite covers of colored graphs highlights
the connection between these sparse colored families and the well-studied matroidal
$(k,\ell)$-sparse families.
\end{normalsize}
\end{abstract}

\section{Introduction} \seclab{intro}
Let $G=(V,E)$ be a finite directed graph, let $\Gamma$ be a group, and let $\bgamma=(\gamma_{ij})$ be
an assignment of a ``color'' $\gamma_{ij}\in \Gamma$.  The tuple $(G,\bgamma)$ is called a
\emph{colored graph}%
\footnote{Colored graphs are also known as ``gain graphs'' \cite{Z98}.}
.  In this paper, $\Gamma$ will always be one of the abelian groups: $\Z/p\Z$, $\Z$,
$\Z/p\Z\times \Z/q\Z$, or $\Z^2$.  For these $\Gamma$, there is a well-defined homomorphism $\rho$
from the cycle space $\HH_1(G,\Z)$ to $\Gamma$, which we describe in \secref{symmetric}.  The
number of linearly independent elements in the the image of $\rho$ on a subgraph $G'$ of $G$ is an invariant of
$G'$ which we define to to be the \emph{$\rho$-rank of a subgraph}.  In this note, we study colored
graphs defined by a hereditary sparsity property that depends on the $\rho$-rank.  These
generalize the well-studied $(k,\ell)$-sparse graphs \cite{LS08,FS07}, which are
defined by the condition ``$m'\le kn'-\ell$'' for all subgraphs.

\subsection{Sparse colored graphs}
Let $(G,\bgamma)$ be a colored graph, with $n$ vertices and $m$ edges.  Further,
let $G'$ be an edge-induced subgraph with $n'$ vertices, $m'$ edges, $\rho$-rank $r$,
and $c'_i$ connected components with $\rho$-rank $i$ ($i$ will always be in $\{0,1,2\}$).
Then $(G,\bgamma)$ is defined to be \emph{Ross-sparse}\footnote{An equivalent definition is due to Elissa Ross.}
if, for all edge-induced subgraphs
\begin{equation}\eqlab{ross-sparse}
m' \le 2n' - 3c'_0 -2(c'_1 +c'_2)
\end{equation}
it is \emph{cone-Laman-sparse} if, for all subgraphs
\begin{equation}\eqlab{cone-laman-sparse}
m'\le 2n' - 3c'_0 - c'_1 - c'_2
\end{equation}
and it is \emph{cylinder-Laman-sparse} if, for all subgraphs
\begin{equation}\eqlab{cylinder-laman-sparse}
m'\le 2n' + r - 3c'_0 - 2(c'_1 + c'_2)
\end{equation}
In, in addition \eqref{ross-sparse} (resp. \eqref{cone-laman-sparse}, \eqref{cylinder-laman-sparse})
hold with equality on the whole graph, then $(G,\bgamma)$ is a \emph{Ross-graph},
(resp. \emph{cone-Laman graph}, \emph{cylinder-Laman graph}).

\subsection{Inductive characterization}
These families can be characterized as the graphs generated
from a fixed base by a sequence of several inductive moves.  The moves we use are defined in
\secref{moves} and illustrated in \figref{moves}.
\begin{theorem}\theolab{main}
A colored graph $(G,\bgamma)$ with:
\begin{itemize}
\item $\Z^2$ colors is a Ross-graph if and only if it can be constructed from
a base graph as in \figref{ross-base} using the moves \textbf{(H1c)} and
\textbf{(H2c)} \cite{R11}.
\item $\Z/p\Z$ colors, with $p$ an odd prime,
is a cone-Laman graph if and only if it can be constructed from
a base graph as in \figref{cone-base} using the moves \textbf{(H1c)},
\textbf{(H1c$'$)}, and \textbf{(H2c)}.
\item $\Z$  colors is a cylinder-Laman-graph if and only if it can be constructed from
a base graph as in \figref{cylinder-base} using the moves \textbf{(H1c)} and
\textbf{(H2c)}.
\end{itemize}
\end{theorem}
Results like this are known as \emph{Henneberg constructions}, since they generalize a
classical technique from \cite{H11} to all matroidal $(k,\ell)$-sparse graphs
\cite{FS07,LS08}.

\subsection{Interpretation of the colored Henneberg 2 move}
The somewhat technical nature of the colored-Henneberg move
\textbf{(H2c)} has a more natural interpretation.  Immerse the colored graph $(G,\bgamma)$ in
$\R^2/\Gamma$ with geodesic edges selected by the colors.  The colored Henneberg move \textbf{(H2c)}
then corresponds to putting the new vertex $n$ on the edge $ij$ that is being split and connecting $n$ to
its other neighbor $k$ using the geodesic specified by the color on the new edge $ik$.  This is a stronger
statement than simply saying that there is \emph{some} choice of coloring for the new edges would preserve
the desired sparsity property.

\subsection{Combinatorial rigidity motivation}
All the families of colored graphs described above arise from instances of the following
geometric problem.  A \emph{$\Gamma$-framework} is a planar structure made of fixed-length bars
connected by joints with full rotational freedom; additionally, it is symmetric under a
representation of $\Gamma$ by Euclidean motions of the plane, which induces a free $\Gamma$-action
by automorphisms on the graph $\tilde{G}$ that has as its edges the bars.  The allowed motions
preserve the length and connectivity of the bars \emph{and} symmetry, but not necessarily the
representation of $\Gamma$.  A \emph{$\Gamma$-framework} is \emph{rigid} when the allowed
motions are all Euclidean isometries and otherwise \emph{flexible}.
Generically, rigidity and flexibility are properties of the
colored graph $G$ that encodes $\tilde{G}$, and the ``Maxwell-Laman question'' (cf. \cite{L70,M64})
is to characterize the combinatorial types of generic, minimally rigid frameworks.

Justin Malestein and the author solved this problem for: \emph{periodic frameworks}
\cite{MT10}, where $\Gamma$ is $\Z^2$ acting by translations with ``flexible representation'' of the
translation lattice;
for \emph{crystallographic frameworks}, where $\Gamma$ is generated by translations and a rotation of order
$2,3,4$ or $6$ \cite{MT11}.  For the periodic case, the minimally rigid \emph{colored-Laman graphs} are defined,
using the notation above, by the counts
\begin{equation}\eqlab{colored-laman-sparse}
m'\le 2n' + \max\{2r-1,0\} - 3c_0 - 2(c_1 + c_2)
\end{equation}
At the time \eqref{colored-laman-sparse} had not been conjectured, nor, to the best of our
knowledge, had matroidal families defined by counts of this form appeared in the combinatorial
literature.  The geometric idea leading to \eqref{colored-laman-sparse} is that a sub-framework
that ``sees'' $r$ flexible periods has $2r-1$ non-trivial degrees of freedom from the lattice
representation, $2n'$ from the coordinates of the vertices, and each connected component has
either two or three ``trivial'' motions commuting with any fixed lattice representation.\footnote{The paper \cite{MT12}
explains this geometric derivation, and its generalization to other groups, in more detail.}

The colored graph families under consideration here also correspond to generic minimally rigid
frameworks in different forced-symmetric models: Ross graphs for \emph{fixed-lattice periodic
frameworks} \cite{R11,MT10}, which are periodic frameworks where the translation lattice is fixed;
cone-Laman graphs for \emph{cone frameworks} \cite{MT11}, where the symmetry group is a
finite-order rotation around the origin; and cylinder-Laman graphs for \emph{cylinder frameworks}
\cite{MT10,MT12} which are periodic with one flexible period.

\subsection{Novelty}
The combinatorial steps in \cite{MT10,MT11} rely on the ``edge-doubling trick'' of
Lovász \cite{LY82} and Recski \cite{R84} and then decompositions obtained by submodular function theory \cite{ER66}.
Seeing as the colored graph families under consideration arise in a planar rigidity setting, it
is natural to ask what the connection they have to the well-studied $(2,3)$-sparse graphs
(shortly \emph{Laman graphs}) characterizing the minimally rigid planar frameworks \cite{L70}.

It is not hard to see that the $\rho$-rank zero subgraphs must be $(2,3)$-sparse.  On the
other hand, the proof method employed here is based around the following proposition that
characterizes a cone-Laman graph in terms of its symmetric cover:
\begin{prop}[\liftprop]\proplab{cone-laman-lift}
Let $p$ be an odd prime, and let $(G,\bgamma)$ be a $\Z/p\Z$-colored graph with $n$ vertices and
$2n-1$ edges.  Then $(G,\bgamma)$ is a cone-Laman graph if an only if its symmetric lift
$(\tilde{G},\varphi)$ is Laman-sparse.
\end{prop}
This connection between colored sparsity and $(k,\ell)$-sparse covers is, to me, as interesting
as \theoref{main}.  As an algorithmic consequence, if $p$ is small relative to the number of
vertices $n$, the algorithmic rigidity questions, as defined in \cite{LS08}, for cone-Laman graphs
can all be solved in $O(n^2)$ time with the \emph{pebble game} \cite{LS08,BJ03}.  It was in this
context that the specialization of \propref{cone-laman-lift} $p=3$ was first observed \cite{BHMT11}.

\propref{cone-laman-lift} doesn't extend, naively\footnote{This had been observed heuristically in
\cite{GH03}, and is discussed in more detail in \cite[Section 19.5]{MT10}.}, at least, to  the
colored-Laman graphs of \cite{MT10} or the $\Gamma$-colored-Laman graphs of \cite{MT11}.
Thus, we also obtain a distinction between the cone-Laman-sparse
colored graph families and the more general ones introduced to understand periodic and crystallographic
frameworks.  Finding the ``right'' generalization of \propref{cone-laman-lift} would be very interesting.

\subsection{Roadmap to the proof of \theoref{main}}
\propref{cone-laman-lift}  allows us to study
the combinatorial structure of cone-Laman graphs via the symmetric
lift $\tilde{G}$. This allows us to apply the entire theory of Laman-sparse
graphs apply.  Colored Henneberg moves on a cone-Laman graph $G$ correspond
to ``symmetrized groups'' of uncolored Henneberg moves on $\tilde{G}$.
The idea of symmetrizing Henneberg moves is not new \cite{S10}, but the approach
taken here is.  The difficult step (\lemref{cone-h2-reverse}) is to show that,
after removing an entire vertex \emph{orbit} in $\tilde{G}$, an entire edge \emph{orbit}
may be added to the remaining graph while maintaining Laman-sparsity.  The
proof makes use of a new circuit-elimination argument
(\propref{cyclic-eliminate}) that avoids a complicated cases analysis.

\subsection{Notations}
When dealing with the families of (uncolored, finite) $(k,\ell)$-sparse graphs \cite{LS08}, we
adopt the following conventions:  \emph{$(k,\ell)$-circuits} are minimal violations
of sparsity; a graph is \emph{$(k,\ell)$-spanning} if it contains a spanning $(k,\ell)$-graph;
a \emph{$(k,\ell)$-block} is a subgraph that is a $(k,\ell)$-graph; and a \emph{$(k,\ell)$-basis} is a maximal
$(k,\ell)$-sparse subgraph.  As is standard, we refer to $(2,3)$-sparse graphs as
\emph{Laman graphs}.

Colored graphs are directed, so an edge $ij$ means a directed edge from $i$ to $j$.  Their symmetric
lifts are undirected, so the order of the vertices in an edge of the lift doesn't indicated
orientation.

\subsection{Acknowledgements}
This work is supported by the European Research Council under the European Union’s Seventh Framework
Programme (FP7/2007-2013) / ERC grant agreement no 247029-SDModels.  I first circulated these
results at the Fields Institute's Rigidity and Rigidity and Symmetry workshops in the Fall of 2011,
and I want to thank the Fields Institute and workshop organizers for their hospitality.

\section{Colored graphs and their lifts}\seclab{symmetric}
In this short section, we quickly review some facts about colored graphs.
Colored graphs are an efficient encoding of a (not-necessarily finite, undirected) graph
$\tilde{G}=(\tilde{V},\tilde{E})$ with a free $\Gamma$-action $\varphi$ acting by automorphisms
with finite quotient. We call the tuple $(\tilde{G},\varphi)$ a \emph{symmetric graph}.

\subsection{Symmetric graphs and colored quotients}
A straightforward specialization of covering space theory, which is given in detail in
our paper \cite[Section 9]{MT11} with Justin Malestein, links colored and symmetric graphs: each
colored graph \emph{lifts} canonically to a \emph{symmetric cover} $\tilde{G}$, and,
after selecting representatives for the $\Gamma$-orbits of vertices, each symmetric
graph  $(\tilde{G},\varphi)$ determines a colored graph $(G,\bgamma)$, with undirected
graph underlying $G$ being $\tilde{G}/\Gamma$.

\subsection{The homomorphism $\rho$}
While the choice of colored quotient is not canonical, the rank of the image of
the induced homomorphism $\rho : \HH_1(G,\Z)\to \Gamma$ is, which justifies the
use of colored graphs in situations where the natural definition is in
terms of symmetric graphs.  To compute $\rho$ on a cycle $C$, we traverse
$C$ in some direction, adding up the colors on the edges traversed forwards and
subtracting the colors on edges traversed backwards.
Since $\rho$ is linear on $\HH_1(G,\Z)$, it is
determined by its images on any cycle basis of $G$.  In particular,
the fundamental cycles of any spanning forest $F$ of $G$ are a cycle basis,
and so we can always assume that the colors are zero on $F$.

\subsection{Edge orbits and colors}
If $(\tilde{G},\varphi)$ is a symmetric graph with colored
quotient $(G,\bgamma)$, we denote vertices  the fiber over a vertex $i\in V(G)$
is given by $\tilde{i}_{\gamma}$, $\gamma\in \Gamma$, and the fiber over
a (colored, oriented) edge $ij\in V(G)$ by $\tilde{i}_{\gamma}\tilde{j}_{\gamma_{ij}+\gamma}$,
for $\gamma\in \Gamma$.

From the definition of the symmetric cover, we see that:
\begin{lemma}\lemlab{colors}
Let $(G,\bgamma)$ be a colored graph, and let $(\tilde{G},\varphi)$ be its lift.  Then
an edge $\tilde{i}_{\gamma}\tilde{j}_{\gamma'}$ is in $(\tilde{G},\varphi)$ if and
only if either there is an (oriented) edge $ij$ in $G$ with color $\gamma' - \gamma$ or
an oriented edge $ji$ in $G$ with color $\gamma-\gamma'$.\eop
\end{lemma}

\subsection{Subgraph orbits}
We also require (cf. \cite[Corollary 15]{BHMT11}):
\begin{lemma}\lemlab{connected-lift}
Let $p$ be an odd prime, and let $(G,\bgamma)$ be a $\Z/p\Z$-colored  graph and let $G'$ be a connected
subgraph, and let $\tilde{G}'$ be the lift of $G'$.  Then $\tilde{G}'$ is connected if and only if
$G'$ has non-zero $\rho$-rank.
\end{lemma}
\begin{proof}
By selecting a spanning forest $F$ of $G$ that is also a spanning forest of $G'$, we may assume that
$G'$ is, in fact, all of $G$.  If $(G,\bgamma)$ has $\rho$-rank zero, then, w.l.o.g., we may assume all
the colors are zero, and the Proposition follows from the construction of the lift.  Otherwise,
we can pick a spanning tree $T$ of $G$ and some addition edge $ij$, so that the fundamental cycle $C$
of $ij$ in $T$ has non-trivial $\rho$-image.  The lift of $C$ must be a collection of $t$ cycles,
with $t$ dividing $p$, which is possible only if $t$ is $1$ or $p$; the latter is only possible if
$ij$ is a self-loop with trivial $\rho$-image, contradicting how it was selected.  Thus, the lift of
$(G,\bgamma)$ contains $p$ copies of $T$, connected by a cycle covering $C$.
\end{proof}
A consequence is that if $a$ and $b$ have a common neighbor $i$, the element of $\Gamma$ obtained
by the oriented sum (in the sense of the definition of $\rho$) of the edge colors
on the path from $a$ to $b$ via $i$ can be ``read off'' from the neighbors of the vertices in the fiber over $i$
in the lift $\tilde{G}$.
\begin{lemma}\lemlab{colors2}
Let $(G,\bgamma)$ be a $\Z/p\Z$ colored graph, and let $(\tilde{G},\varphi)$ be its lift.  Let
$i$ be a vertex with neighbors $a$ and $b$ in $G$, and let $\eta$ be the sum of the
colors along the path $a$--$i$--$b$ as defined for the map $\rho$.  Then in the
symmetric lift $\tilde{G}$, the neighbors $\tilde{a}_{\gamma}$ and $\tilde{b}_{\gamma'}$ of
any vertex $\tilde{i}_{\delta}$ in the fiber over $i$ satisfy
$\eta = \gamma' - \gamma$.
\end{lemma}
\begin{proof}
Add the oriented edge $ab$ to $(G,\bgamma)$ with color $\eta$ to get a colored
graph $(H,\bgamma)$.  The subgraph with the path from $a$ to $b$ via $i$ and $ab$ has, by
construction, $\rho$-rank zero.  Thus,
\lemref{connected-lift} says that its lift is $p$ vertex disjoint triangles.  The
lemma follows from applying \lemref{colors} to the fiber over $ab$ in the lift
$\tilde{H}$.
\end{proof}

\section{The lift of a cone-Laman graph}\seclab{cone-lift}
This next proposition, which is a generalization of \cite[Lemma 6]{BHMT11},
is our basic technical tool.

\liftprop
\begin{proof}
We prove the contrapositive in both directions.  First suppose that $(G,\bgamma)$ is not
cone-Laman sparse.  Minimal violations (i.e., cone-Laman-circuits) come in two
types: Laman-circuits with trivial $\rho$-image and subgraphs with non-trivial image, $n'$
vertices and $2n'$ edges.  \lemref{connected-lift} says that the first type lifts to $k$
copies of itself, blocking Laman-sparsity in the lift $(\tilde{G},\varphi)$.  The
second type lifts to a subgraph of  $(\tilde{G},\varphi)$ that has $pn'$ vertices and
$2pn'$ edges, which is certainly not Laman-sparse.

Now we suppose that the lift $(\tilde{G},\varphi)$ spans some Laman-circuit $H$.
Denote by $H_\gamma$ the image of $H$ under $\varphi(\gamma)$, so that the orbit $\tilde{H}$
of $H$ is the union of the $H_\gamma$.  If the $H_\gamma$ are all disconnected from each other,
then $\tilde{H}$ is, by \lemref{connected-lift}, the lift of a Laman-circuit with trivial $\rho$-image.
Otherwise, again using \lemref{connected-lift}, $\tilde{H}$ is a graph on $n'$ vertices
made by gluing $p$ Laman-circuits together in a ring-like fashion along Laman-sparse subgraphs.  Thus, it
has at most $p-3$ Laman degrees of freedom and at least $k$ Laman-dependent edges.  In other
words, a Laman-basis of $\tilde{H}$ has at least $2n' - p$ edges and there are $p$ other
edges, implying that it has at least $2n'$ edges in total.
\end{proof}
Our other technical tool is:
\begin{prop}\proplab{cyclic-eliminate}
Let $p$ be an odd prime, and let  $(\tilde{G},\varphi)$ be a symmetric graph with a
$\Z/p\Z$-action $\varphi$. Suppose that $H$ is a Laman-circuit in $\tilde{G}$, and suppose that, for some
$\gamma'\in \Z/p\Z$, $\varphi(\gamma')\cdot H$ and $H$ intersect on an edge $ij$.  Then there is a
Laman-circuit in $\tilde{G}$ that goes through one edge in the orbit of $ij$.
\end{prop}
\begin{proof}
As in the proof of \propref{cone-laman-lift}, denote by $H_\gamma$ in the images of $H$ under $\varphi(\gamma)$
and the whole orbit by $\tilde{H}$ and adopt similar notation for $ij$.  Because $k$ is prime, $\gamma'$ has order $k$,
so we may, w.l.o.g., assume $\gamma=1$.  It follows that $H_\gamma\cap H_{\gamma+1}$  is never empty, so we may
assume, w.l.o.g., that $(ij)_\gamma\in H_\gamma\cap H_{\gamma+1}$.  Since Laman-circuits are Laman-spanning,
and $\tilde{H}$ is made by gluing Laman-circuits (the $H_\gamma$) along at least two vertices (the endpoints
of $(ij)_\gamma$), it follows from \cite[Theorem 5]{LS08} that $\tilde{H}$ is Laman-spanning as well.

The Proposition will follow from showing that $\tilde{H}$ has a Laman-basis $\tilde{L}$ that doesn't contain any edge
in the orbit of $ij$, since the fundamental circuit of $ij$ in $\tilde{L}$ produces the desired circuit.  We do this
by refining the argument above.  Let $L_1 = H - ij$.  Since $H$ is a Laman-circuit, $L_1$ is a Laman-graph, and it
contains $(ij)_1$.  Since $H_1$ is a Laman-circuit, $H'_1 = H_1 - (ij)_1$ is Laman-spanning, and thus, so is
$L_1\cup H'_1$.  Because $(ij)_1$ is in the span of the Laman-block $H'_1$, and, if $ij$ is
present in $H_1$, it is in the span of the  Laman-block $L_1$,  $L_2$, $L_1\cup H'_1$
has a Laman-basis  $L_2$ that does not contain $ij$ or $(ij)_1$.  Repeating this process $p$ times, we obtain the
desired $\tilde{L}$.
\end{proof}

\paragraph{Remark} The proof of \propref{cyclic-eliminate} is written from the perspective of bases, but
it can be argued directly from the perspective of circuits as well, obtaining a slightly different conclusion.
We eliminate $ij$ from the intersection of $H$ and $H_1$ to obtain a circuit $C_1$ in $\tilde{G}$ that does not go
through $ij$ but does contain $(ij)_1$.  Iterating we obtain a family of circuits $C_2,C_3,\ldots,C_t$
such that $C_t'$ does not contain $(ij)_\gamma$ for $\gamma<t'$.  The process either ends at some $t<p$,
yielding a circuit disjoint from the orbit of $ij$ or at $C_p$, which contains only $(ij)_{k-1}$ from
the orbit of $ij$.

\section{Colored Henneberg moves}\seclab{moves}
In this section we define the Henneberg moves that we will use, and the base graphs for each sparsity type.

\subsection{The uncolored Henneberg moves}
If we forget about the colors, these are just the generalized Henneberg moves that can be found in
\cite{LS08,FS07}; we will call these
\emph{uncolored Henneberg moves} \textbf{(H1)}, \textbf{(H1$'$)}, and \textbf{(H2)} to distinguish them from the colored
moves defined here.  The following facts may be found in \cite{LS08,FS07}:
\begin{prop}[][{\cite{LS08,FS07}}]\proplab{uncolored-moves}
The uncolored Henneberg moves:
\begin{itemize}
\item Preserve $(2,1)$- and $(2,2)$-sparsity.
\item Perserve Laman-sparsity when the neighbors of the new vertex are all distinct.
\item Generate exactly $(2,2)$-graphs, starting from a doubled edge.
\item Generate exactly $(2,1)$-graphs, starting from a vertex with a single self-loop.
\end{itemize}
\end{prop}
\begin{figure}[htbp]
\centering
\subfigure[]{\includegraphics[width=0.25\textwidth]{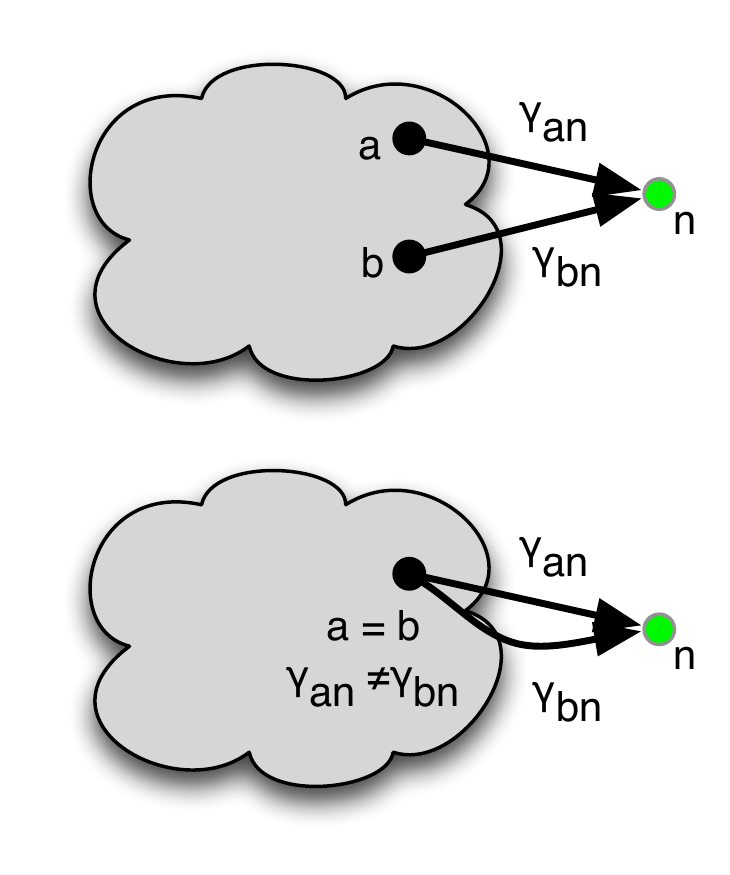}}
\subfigure[]{\includegraphics[width=0.25\textwidth]{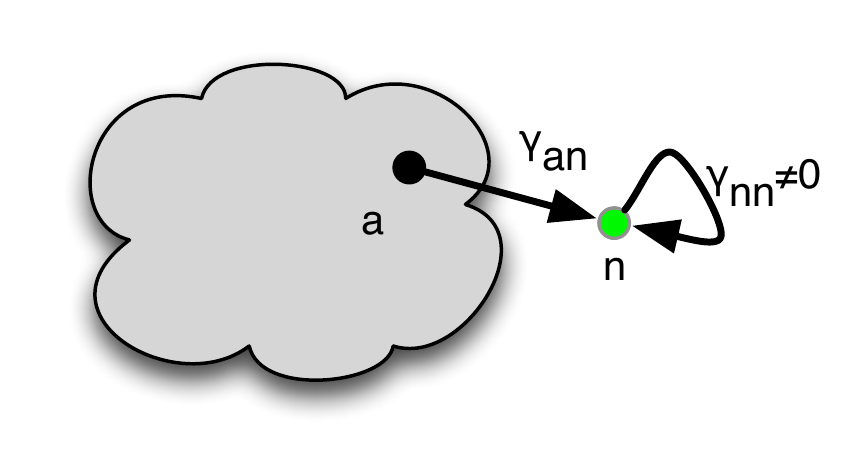}}
\subfigure[]{\includegraphics[width=0.25\textwidth]{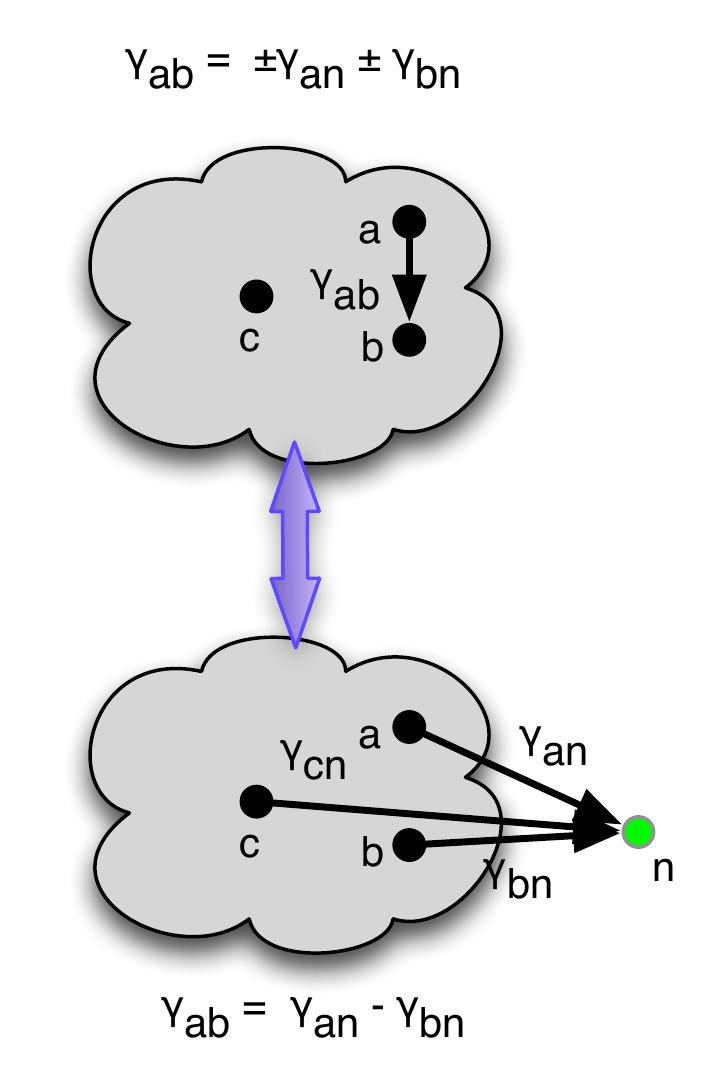}}
\caption{The colored Henneberg moves: (a) \textbf{(H1c)}; (b) \textbf{(H1c$'$)}; (c) \textbf{(H2c)}.}
\label{fig:moves}
\end{figure}

\subsection{Forward and reverse moves}
All the moves have forward and reverse directions.  In each direction, we specify the allowed
orientations and colors of any new edges.  The forward moves can always be applied, while the
reverse ones work only on a vertex of the appropriate degree.

\subsection{The (H1c) and (H1c$'$) moves}
We start with the simpler two moves.  These involve adding one new vertex $n$ and two new edges.  For \textbf{(H1c)},
$n$ is connected to the existing graph by two edges $an$ and $bn$; by convention we orient them into $n$, and, require that
if $a=b$, the colors are different.  The reverse move just removes a degree two vertex.  The move \textbf{(H1c$'$)} prime,
which we give the suggestive mnemonic ``lollipop move'', connects the new vertex $n$ to the existing graph by one new
edge $an$ with arbitrary color and adds a self-loop on $n$ with non-zero color.  The reverse move simply removes
a vertex incident on one self-loop and one other edge.
\begin{figure}[htbp]
\centering
\subfigure[Cone-Laman, $\gamma\in \Z/k\Z$]{\includegraphics[width=0.15\textwidth]{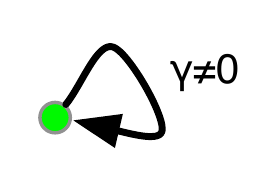}\label{fig:cone-base}}
\hfill
\subfigure[Cylinder-Laman, $\gamma\in \Z$]{\includegraphics[width=0.15\textwidth]{cone-laman-base}\label{fig:cylinder-base}}
\hfill
\subfigure[Ross graphs,
$\gamma_{12}, \gamma_{(12)'}\in \Z^2$]{\includegraphics[width=0.15\textwidth]{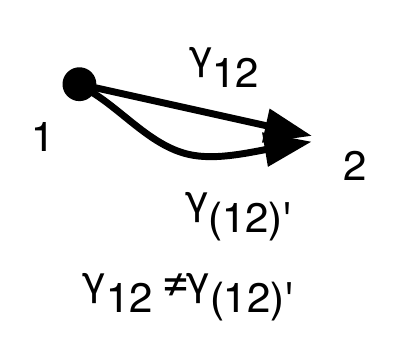}\label{fig:ross-base}}
\caption{The base cases.}
\label{fig:bases}
\end{figure}
\paragraph{Remark} The convention regarding the orientation in the forward direction doesn't impose a
restriction, since $\rho$-rank is preserved if we change the orientation of an edge and the
sign of the color on the edge at the same time.
\subsection{The (H2c) move}
The \textbf{(H2c)} move, which adds a new vertex $n$, removes one edge, and adds three new ones is slightly more
complicated.  Let $ab$ be an edge with color $\gamma_{ab}$, and let $c$ be some other vertex.  Note that $a$, $b$, and
$c$, are not necessarily distinct.  The forward \textbf{(H2c)} move removes the edge $ab$ and replaces it with edges
$an$ and $nb$ colored such that $\gamma_{an} - \gamma_{bn} = \gamma_{ab}$; an edge $ac$ with arbitrary color is also
added.  If any of $a$, $b$, and $c$ are the same, we further require that any parallel edges added have pairwise
different colors.

The reverse direction is slightly more complicated.  We don't have control over the orientation of the
edges at the degree $3$ vertex, and there are, potentially, several possibilities of the endpoints of the
edge to put back, as well as a number of potential colors.
We start with a degree-three vertex $n$, with neighbors $a$, $b$, and $c$, which, again, may not be distinct.
A reverse \textbf{(H2c)} move removes $n$ and adds back $ab$ (resp. $ac$, $bc$) with some orientation and
color that is the oriented sum, in the sense of the map $\rho$'s definition, of the oriented path short-circuited
by $ab$ (resp. $ac$, $bc$).
\paragraph{Remark}  In the proof that the reverse \textbf{(H2c)} move preserves cone-Laman sparsity,
we will see that the color of the replacement edge is determined by the correspondence found
in \lemref{colors}.

\subsection{Base graphs}
We also have to specify the base cases of our induction.  These are shown in \figref{bases}.

\section{\theoref{main} for cone-Laman graphs}\seclab{cone-laman-proof}
With the definition of the moves complete, we are in a position to
prove \theoref{main} for cone-Laman graphs.  This occupies the rest
of the section.  To set the notation, let $p$ be an odd prime and
let $(G,\bgamma)$ be a cone-Laman graph with $\Z/p\Z$ colors.  The
new vertex will be $n$.

\subsection{Applicability of the colored Henneberg moves}
Because the colors come from $\Z/p\Z$, the $\rho$-rank of any subgraph is
always zero or one.  Since a cone-Laman graph has $n$ vertices and $2n-1$ edges,
there is always a vertex of degree two or three.  Thus, we need only to
check that the moves defined in \secref{moves} preserve the cone-Laman
property in the forward and reverse directions.

\subsection{The base case}
It is readily seen that any of the claimed base cases is a cone-Laman graph. \eop

\subsection{Colored Henneberg moves and the symmetric lift}
We may interpret a colored Henneberg \textbf{(H1c)} or \textbf{(H2c)} move applied to $G$ as a
group of $p$ uncolored Henneberg moves applied to $\tilde{G}$.

\begin{lemma}\lemlab{h1c-lift}
Let $(H,\bgamma)$ be the colored graph obtained from $(G,\bgamma)$
by applying an \textbf{(H1c)} move that adds a new vertex $n$ and
edges $an$ and $bn$ with colors $\gamma_{an}$ and $\gamma_bn$.  Then
the symmetric lift $\tilde{H}$ is obtained from the lift $\tilde{G}$
by applying $p$ \textbf{(H1)} moves.
\end{lemma}
\begin{proof}
The degree of the vertices in the fiber over the new vertex $n$ are all two,
and the uncolored move \textbf{(H1)} adds a degree two vertex.
\end{proof}

\begin{lemma}\lemlab{h2c-lift}
Let $(H,\bgamma)$ be the colored graph obtained from $(G,\bgamma)$
by applying an \textbf{(H2c)} move that adds a new vertex $n$,
removes an edge $ab$ with color $\gamma_{ab}$, and adds new edges
$an$, $bn$, and $cn$ with colors such that $\gamma_{an} - \gamma_{bn} = \gamma_{ab}$.
Then the symmetric lift $\tilde{H}$ is obtained from the lift $\tilde{G}$
by applying $p$ \textbf{(H2)} moves.
\end{lemma}
\begin{proof}
The degree of the vertices in the fiber over the new vertex $n$ are all three.
\lemref{colors2} says that, for each $\gamma\in\Z/k\Z$ the
neighbors of $\tilde{n}_{\gamma}$ in $\tilde{H}$
are determined by the colors $\gamma_{an}$, $\gamma_{bn}$, and $\gamma_{cn}$ and that the neighbors
of $\tilde{n}_\gamma$ in the fibers over $a$ and $b$ are endpoints of an edge in the fiber over $ab$.
This determines the data specifying an \textbf{(H2)} move for each edge in the fiber over $ab$.
\end{proof}

\subsection{The forward moves}
Now we check that the forward moves preserve the cone-Laman property.  We will do this
using the interpretation of the colored moves in terms of the lift $\tilde{G}$
and uncolored moves and \propref{cone-laman-lift}.

\begin{lemma}\lemlab{cone-h1-forward}
The \textbf{(H1c)} move, applied to $(G,\bgamma)$, results in a cone-Laman graph.
\end{lemma}
\begin{proof}
Let $(H,\bgamma)$ be the graph obtained after the move.  The requirement that
if the neighbors of the new vertex $n$ are not distinct that the new
edges have different colors says that, in the lift $\tilde{H}$, the neighbors
of any vertex in the fiber over the new vertex $n$ are distinct.  \propref{cone-laman-lift},
\lemref{h1c-lift}, and \propref{uncolored-moves} imply that $\tilde{H}$ is Laman-sparse.
Since $H$ has $2n-1$ edges, the lemma follows.
\end{proof}
The proof of the next Lemma is nearly identical, with \lemref{h2c-lift} replacing \lemref{h1c-lift},
so we omit it.
\begin{lemma}\lemlab{cone-h2-forward}
The \textbf{(H2c)} move, applied to $(G,\bgamma)$, results in a cone-Laman graph.\eop
\end{lemma}
The lollipop move \textbf{(H1c$'$)} requires slightly more careful consideration
of the lift $\tilde{H}$.  There is no version of \lemref{h2c-lift}
for this move, because vertices in the fiber over the new vertex are neighbors
with each other.
\begin{lemma}\lemlab{cone-h1cp-forward}
The \textbf{(H1c$'$)} move, applied to $(G,\bgamma)$, results in a cone-Laman graph.
\end{lemma}
\begin{proof}
We consider the lift $\tilde{H}$.  Any subset of $t$ vertices in the fiber over the new vertex $n$
spans at most $t$ edges and connects to the rest of $\tilde{H}$ with exactly $t$ edges.  Thus,
for any $V'\subset V(\tilde{H})$ on $n'$ vertices not in the fiber over $n$ and $t$ in the
fiber over $n$, the number of edges induced by $V'$ is bounded by $2n' - 3 + 2t = 2|V'| - 3$,
since $\tilde{G}$ is Laman-sparse by \propref{cone-laman-lift}.
\end{proof}

\paragraph{Remark}  The distinction
between \textbf{(H2c)} and \textbf{(H1c$'$)} is implicit in \cite{S10}.

\paragraph{Remark} With a slightly more delicate argument, using some structural
results from \cite{BHMT11,MT12}, we can show that the lemmas above hold even when
$p$ isn't prime by working with the colored graph directly.
Since we don't need the extra generality, we omit the proof.

\subsection{The reverse moves}
To complete the proof, we check that the reverse moves also preserve the
cone-Laman property.  In light of \propref{cone-laman-lift} and \lemref{h1c-lift},
the following are straightforward.
\begin{lemma}\lemlab{cone-h1-reverse}
The reverse \textbf{(H1c)} move, applied to $(G,\bgamma)$, results in a cone-Laman graph.\eop
\end{lemma}
\begin{lemma}\lemlab{cone-h1p-reverse}
The reverse \textbf{(H1c$'$)} move, applied to $(G,\bgamma)$, results in a cone-Laman graph.\eop
\end{lemma}
The hard step is the reverse \textbf{(H2c)} move, which only says that there is \emph{some}
edge we can put back with locally determined colors and orientation.
\begin{lemma}\lemlab{cone-h2-reverse}
Given any degree-three vertex $i$ in $(G,\bgamma)$ not incident on any self-loop, there is a
reverse \textbf{(H2c)} move, applied to $i$, that results in a cone-Laman graph.
\end{lemma}
\begin{proof}
Let $(\tilde{G},\varphi)$ be the symmetric lift, and let $i$ be a degree three vertex in $G$ with neighbors
$a$, $b$, and $c$. Since we are doing a reverse \textbf{(H2c)} move (and not a lollipop),
$a$, $b$, and $c$ are all different from $i$ (though not necessarily each other).

Let
$\tilde{a}_\alpha$, $\tilde{b}_\beta$, $\tilde{c}_\gamma$ be the neighbors of $\tilde{i}_0$.
\propref{cone-laman-lift} tells us that these vertices are all different from each other, even if
they are in a common orbit.  \lemref{colors2} and \lemref{h2c-lift} tell us that it is sufficient to
show that if we can remove the fiber over $i$ from $\tilde{G}$ and add back the orbit of an edge
between the neighbors of $\tilde{i}_0$, the lemma will follow.  Let $\tilde{H}$ be the
symmetric graph obtained by removing the fiber over $\tilde{i}_0$ from $\tilde{G}$.

\propref{uncolored-moves} implies that there is an edge between some pair of
$\tilde{a}_\alpha$, $\tilde{b}_\beta$, $\tilde{c}_\gamma$ that, when added to $\tilde{H}$,
results in a Laman-sparse graph.  Without loss of generality, this is
$\tilde{a}_\alpha\tilde{b}_\beta$.  The crux of the proof is that we can put back the entire orbit of
$\tilde{a}_\alpha\tilde{b}_\beta$ maintaining Laman-sparsity.  Let $\tilde{H}'$ be the graph $\tilde{H}$
with the orbit of $\tilde{a}_\alpha\tilde{b}_\beta$ added to it.

Suppose, for a contradiction, that $\tilde{H}'$ is not Laman-sparse.
Since $\tilde{H}+\tilde{a}_\alpha\tilde{b}_\beta$
is Laman-sparse, symmetry implies that any Laman-circuit in $\tilde{H}'$ goes through two of the edges
in the orbit of $\tilde{a}_\alpha\tilde{b}_\beta$.
This is the situation from \propref{cyclic-eliminate}, leading to a contradiction: the new
edges were selected so that there are no Laman-circuits through exactly one of them, but such a
circuit is forced by \propref{cyclic-eliminate}.
\end{proof}

\section{\theoref{main} for cylinder-Laman graphs}\seclab{cylinder-laman-proof}
We now turn to cylinder-Laman graphs.  There are two differences, between this case and the cone-Laman
one: we want the colors to come from $\Z$ and we have to check that the two allowed moves
can't generate a graph that is cone-Laman, but not cylinder-Laman.  (It is clear that the lollipop move
\textbf{(H1c$'$)}  does this.)

\subsection{From $\Z/p\Z$ colors to $\Z$ colors}
Instead of trying to replicate \propref{cone-laman-lift} on an infinite $(\tilde{G},\varphi)$,
we instead use the following reduction.
\begin{lemma}
Let $(G,\bgamma)$ be a $\Z$-colored graph.  Then $(G,\bgamma)$ is cone-Laman with $\Z$
colors if and only if it is cone-Laman for $\Z/p\Z$ colors for some sufficiently
large prime $p$.
\end{lemma}
\begin{proof}
Pick the prime $p$ large enough so that the magnitude of the colors arising in reverse steps is strictly
less than $p$.
\end{proof}

\subsection{From cone-Laman to cylinder-Laman}
Cylinder-Laman graphs are characterized by \cite[Theorem 8]{MT12} as cone-Laman
graphs that have a $(2,2)$-spanning underlying graph. This means the only thing to check is:
\begin{lemma}
The \textbf{(H2)} move preserves the property of being $(2,2)$-spanning in the forward and
reverse directions.
\end{lemma}
\begin{proof}
The Tutte-Nash-Williams Theorem \cite{T61,N61} says that a $(2,2)$-spanning graph decomposes into two connected
subgraphs.  Since a degree three vertex will always be a leaf in exactly one of these subgraphs, for \emph{any}
such decomposition, the lemma is clear.
\end{proof}

\section{\theoref{main} for Ross-graphs}\seclab{ross-proof}
Finally, we adapt our technique to Ross graphs, recovering a result of \cite{R11}.

\subsection{The base case}
Checking that the claimed bases are Ross graphs is straightforward.\eop

\subsection{Inductive step}
Since the underlying graphs of Ross graphs are $(2,2)$-graphs \cite[Lemma 4]{BHMT11},
by the Henneberg construction for $(2,2)$-graphs \cite{LS08,FS07}, we just need to
check the \textbf{(H1c)} and \textbf{(H2c)} moves in each direction.  The
proof for \textbf{(H1c)} is identical to that in \secref{cone-laman-proof},
as is the forward direction for \textbf{(H2c)}.

For the reverse direction, we can't directly apply \propref{cone-laman-lift},
unless the colored graph $(G,\bgamma)$ has $\rho$-rank one.  However,
if it does not, we can make the following modification.
\begin{prop}
Let $p$ and $q$ be distinct odd primes, and let $(G,\bgamma)$ be a $\Z/p\Z\times \Z/q\Z$-colored graph
with $2n-1$ edges.  Then $(G,\bgamma)$ is cone-Laman if and only if its lift is Laman-sparse.
\end{prop}
\begin{proof}
\lemref{connected-lift} has the following refinement: the number of connected components in the
lift of $(G,\bgamma)$ is the index of its $\rho$-image in $\Gamma$ \cite[Lemmas 5.4 and 5.5]{MT10}.  In this
case, the only possibilities for the index are $p$, $q$, and $1$.  The rest of the proof of
\propref{cone-laman-lift} then goes through.
\end{proof}
The reductions from Ross-graphs to cone-Laman graphs then goes through using the steps from
\secref{cylinder-laman-proof}.

\section{Conclusions}
We conclude we several questions and potential directions.
\subsection{Hennberg constructions for all cone-Laman graphs}
For cylinder-Laman  graphs \theoref{main}
settles the question of inductive constructions. 	We also gave a new, pleasant, proof of an existing
characterization for Ross graphs.  We note, however, that while cone-Laman
graphs have rigidity characterizations for colors in $\Z/k\Z$  for \emph{any} $k\ge 2$, \theoref{main}
only applies to $\Z/p\Z$.
\begin{question}
Give a Henneberg construction for cone-Laman graphs with colors from any group $\Z/k\Z$.
\end{question}
The main difficulty seems to be when the $\rho$-image of a subgraph has order two.  This
causes \propref{cone-laman-lift} to fail, so it will require a different argument.

\subsection{Ross-circuits and global rigidity}
The most natural application of these Henneberg moves on colored graphs would be in
characterizing \emph{global rigidity} (see e.g., \cite{GHT10,C05}) for
fixed-lattice frameworks \cite[Section 19.1]{MT10}, \cite{R11}.  We are unaware of
a detailed conjecture for the right class of graphs.  However, Bruce Hendrickson's
proof that a globally rigid planar framework must be redundantly rigid \cite{H92}
extends to the fixed-lattice setting.  This tells us that Ross-circuits will play a role.
\begin{question}
Give an inductive characterization of Ross-circuits.
\end{question}
It seems, by analogy with \cite{BJ03}, plausible that \textbf{(H2c)} and $2$-sum
are sufficient.

\subsection{Generalizing \propref{cone-laman-lift}}
A more combinatorial direction relates to generalizing the relationship that holds
between the cone-Laman and Laman matroids.  Recalling \eqref{colored-laman-sparse},
we see that this generalizes the bounding function
``$kn'-\ell$'' from $(k,\ell)$-sparsity in two ways: the dimension of the $\rho$-image determines a
positive adjustment; and the $\rho$-image of each connected component determines a negative adjustment.
We are unaware of families like this having been studied before, though \cite[``Matroid Theorem'']{Z82}
appears to contain part of the story.

In fact, \cite{MT11} extends this idea further, allowing non-abelian groups $\Gamma$ that admit
a kind of ``uniform matroid'' structure \cite[Section 8]{MT11}.  Thus, since the colored graph
families treated here are part of a much more general phenomenon, we ask:
\begin{prop}
Can one characterize sparsity matroids on colored graphs in terms of an appropriate
``matroid lift'' that does not explicitly reference the colors on the base graph?
\end{prop}

\bibliographystyle{plainnat}

\end{document}